\documentclass{article}
\usepackage{graphicx}
\usepackage{newlfont}
\usepackage[colorlinks,linkcolor=blue,citecolor=blue,anchorcolor=blue,bookmarksopen,pdfpagetransition={Wipe}]{hyperref}
\usepackage{amsmath,amsthm,amssymb}
\newtheorem{thm}{Theorem}[section]

\newtheorem{lem}[thm]{Lemma}

\newtheorem{rem}[thm]{Remark}

\numberwithin{equation}{section}
\begin{document}
\title{\textbf{A collocation method for numerical solution of Telegraph equation }}
\author{M. Zarebnia\footnote{Corresponding author: zarebnia@uma.ac.ir}\ and
R. Parvaz\footnote{rparvaz@uma.ac.ir}}
\date{}
\maketitle
\begin{center}
Department of Mathematics, University of Mohaghegh Ardabili,
56199-11367 Ardabil, Iran.
\end{center}
%------------------------------------abstract-----------------------------------------------------------
\begin{abstract}
\indent In this paper, B-spline collocation method is developed for
the solution of one-dimensional hyperbolic telegraph equation. The
convergence of the method is proved. Also the method is applied on
some test examples, and the numerical results have been compared
with the analytical solutions. The $L_\infty$,$L_2$ and Root-Mean-Square
errors (RMS) in the solutions show the efficiency of the method
computationally.
\end{abstract}
%------------------------------------------Keywords-----------------------------------------------
\vskip .3cm
%\indent \textit{\textbf{MSC2010:}} 65D07; 65M15; 65N12; 65Z99.
%-----------------------------------------Introduction-----------------------------------------
%\newpage
\vskip .3cm
\section{Introduction}
\indent \hskip .65cm Hyperbolic partial differential equations are
commonly used in signal analysis for transmission and propagation of
electrical signals \cite{1} and also has applications in other fields
\cite{2,3}. In the present paper, a collocation approach based on quintic
B-spline functions is utilized for the numerical solution of the
one-dimensional hyperbolic telegraph equation. In recent
years, many different methods have been used to estimate the
solution of the one-dimensional hyperbolic telegraph equation; see,
for example, [{\color{blue}4}-{\color{blue}9}]. Consider the second-order linear hyperbolic
partial differential equation in one-space dimension:
\begin{equation}\label{1}
\frac{\partial^2u} {\partial t^2} +2\alpha{\frac{\partial u}
{\partial t}}+{\beta^2}u={\frac{\partial^2u} {\partial
x^2}}+f(x,t) \ ,\ {a\leq x \leq b}\ ,\ {t\geq 0},
\end{equation}
with the initial conditions
\begin{eqnarray}\label{2}
u(x,0)=f_{0}(x), \\\label{3}
{\frac{\partial u}
{\partial t}}(x,0)=f_{1}(x),
\end{eqnarray}
and boundary conditions
\begin{eqnarray}\label{4}
u(a,t)=g_{0}(t), u(b,t)=g_{1}(t),\\\label{5}
\frac{\partial u} {\partial x}(a,t)=g_{2}(t),\frac{\partial u} {\partial x}(b,t)=g_{3}(t),
\end{eqnarray}
where $\ \alpha $ and $\ \beta$ are constants.\\
The balance of this paper is organized as follows. In Section 2,
the quintic B-spline collocation method for the numerical solution
of the one-dimensional hyperbolic telegraph equation is described.
In Section 3 we derive convergence of the
B-spline collocation method.
In Section 4, the results of numerical experiments are presented.
A summary is given at the end of the paper in Section 5.
%*************************************Finite element method***********************************************
\section{Quintic B-spline collocation method } The interval $[a,b]$ is partitioned into a mesh of
uniform length $h:=\frac{b-a}{N}$ by the knots $ x_i,i=0,1,\ldots,N$
such that $x_{i}:=x_{0}+ih$ and $a=x_0< x_1 < x_2<\ldots<x_{N-1}< x_N=b$.
To solve the equation (\ref{1}) by collocation method with quintic B-splines as basis
functions, we define the approximation $U^n(x)$ as following
\begin{equation}\label{6}
U^n(x)=\sum_{i=-2}^{N+2} {c^n_iB_i(x)}.
\end{equation}
where $U^n(x)$ is a shape function that approximates $u(x,t_n)$ for
the time level $t_n=nk$ where $k$ is a time step size.
For each time level $t_n$, the set $ \{ c^n_{-2},c^n_{-1},\ldots $
$,c^n_{N+1},c^n_{N+2} \}$
are unknown real coefficients, which are to be found, and the $B_i(x)$ are the quintic
B-spline functions defined by \cite{9,10}
\begin{equation}\label{7}
B_i(x)=\frac{1}{h^5}\left\{%
\begin{array}{ll}
(x-x_{i-3})^5, &       x\in [x_{i-3}, x_{i-2}),       \\
(x-x_{i-3})^5-6(x-x_{i-2})^5,&      x\in[x_{i-2}, x_{i-1}),       \\
(x-x_{i-3})^5-6(x-x_{i-2})^5+15(x-x_{i-1})^5,&      x\in[x_{i-1}, x_{i}),       \\
(x_{i+3}-x)^5-6(x_{i+2}-x)^5+15(x_{i+1}-x)^5,&      x\in[x_{i}, x_{i+1}),       \\
(x_{i+3}-x)^5-6(x_{i+2}-x)^5,&      x\in[x_{i+1}, x_{i+2}),       \\
(x_{i+3}-x)^5, &       x\in [x_{i+2}, x_{i+3}),       \\
\end{array}%
\right.
\end{equation}
where $B_{-2}, B_{-1}, B_0, B_1,\ldots,B_{N+1},B_{N+2}$ form a
basis over the region $a \leq x \leq b$. The values of $B_i(x)$
and its derivatives may be tabulated as in Table \ref{table:t1}. Using
approximate function (\ref{6}) and Table \ref{table:t1}, we have
\begin{equation}\label{8}
u(x_i,t_n)\approx U^{n}_{i}= c^{n}_{i-2}+26c^{n}_{i-1}
+66c^{n}_{i}+26c^{n}_{i+1}+c^{n}_{i+2},
\end{equation}
\begin{equation}\label{9}
\frac{\partial u}{\partial x}(x_i,t_n)\approx \big(U^{\prime}\big)^{n}_{i} =\frac{1}{h}(-5c^{n}_{i-2}-50c^{n}_{i-1}+50c^{n}_{i+1}+5c^{n}_{i+2}),
\end{equation}
\begin{equation}\label{10}
\frac{\partial^2 u}{\partial x^2}(x_i,t_n)\approx
\big(U^{\prime\prime}\big)^{n}_{i}=\frac{1}{h^2}(20c^{n}_{i-2}
+40c^{n}_{i-1}-120c^{n}_{i}+40c^{n}_{i+1}+20c^{n}_{i+2}).
\end{equation}
\begin{center}
\begin{table}[ht!]
\caption{$B_i $,$B^{'}_i $  and $B^{''}_i $ at the node points.} % title of Table
\label{table:t1} % is used to refer this table in the text
\centering % used for centering table
\begin{tabular}{l l l l l l l l l l} % centered columns (4 columns)
\hline
$x$ & $x_{i-3}$ & $x_{i-2}$& $x_{i-1}$& $x_{i}$& $ x_{i+1}$& $ x_{i+2}$& $ x_{i+3}$\\
\hline
\hline
$B_i(x)$ & 0 &1&26&66&26&1&0\\
$hB'_i(x)$ & 0 &5&50&0&-50&-5&0\\
$h^2B''_i(x)$ & 0 &20&40&-120&40&20&0\\
\hline
\hline
\end{tabular}
\end{table}
\end{center}
To apply the proposed method, discretizing the time derivative in
the usual finite difference way, with using following finite difference formulae \cite{NI}, we can write:
\begin{equation}\label{11}
\big(\frac{\partial^2u}{\partial t^2}\big)^{n}_{i}\approx\frac{u^{n+1}_{i}-2u^{n}_{i}+u^{n-1}_{i}}{\Gamma(k)^2},
\end{equation}
\begin{equation}\label{12}
\big(\frac{\partial u}{\partial t}\big)^{n}_{i}\approx\frac{u^{n+1}_i-u^{n-1}_i}{2\Gamma(k)},
\end{equation}
\begin{equation}\label{14}
\big(\frac{\partial^2u}{\partial x^2}\big)^{n}_{i}\approx\frac{\big(\frac{\partial^2u}{\partial x^2}\big)^{n-1}_{i}+
\big(\frac{\partial^2u}{\partial x^2}\big)^{n+1}_{i}}{2},
\end{equation}
where $k$ is a time step size, $\big(\frac{\partial^2
u}{\partial x^2}\big)^n_i:=\frac{\partial^2 u}{\partial
x^2}(x_i,t_{n})$, $u^n_i:=u(x_i,t_{n})$ and
$\Gamma(k)$ is a selected function of $k$
satisfying the following equation
\begin{equation}\label{15}
\Gamma(k)^2=(k)^2(1+\mathcal{O}(k)^j),\hskip .45cm j=0,1,\ldots.
\end{equation}
In the numerical computations, we applied the following two
possible choices for $\Gamma(k)$ to improve the accuracy:
$k$ and $2\sin( \frac{k}{2})$. Hence (\ref{1}) can
be written as:
\begin{equation}\label{16}
\frac{u^{n+1}_{i}-2u^{n}_{i}+u^{n-1}_{i}}{\Gamma(k)^2}+ 2
\alpha \frac{u^{n+1}_{i}-u^{n-1}_{i}}{2\Gamma(k)}+ \beta
^2u^{n}_{i}=\frac{\big(\frac{\partial^2u}{\partial x^2}\big)^{n-1}_{i}+\big(
\frac{\partial^2 u}{\partial x^2}\big)^{n+1}_{i}}{2}+f(x_{i},t_n).
\end{equation}
Rearranging the terms and simplifying we get
\begin{equation}\label{17}
vu_{i}^{n+1}+w\big(\frac{\partial^2u}{\partial x^2}\big)_{i}^{n+1}
=\Phi^n(x_i),
\end{equation}
where
\begin{equation}\label{18}
\Phi^n(x_i):=\Big(2-\big(\beta \Gamma(k)\big)^2\Big)u^{n}_i+\big(\alpha \Gamma(k)-1\big)u^{n-1}_i+
\frac{\Gamma(k)^2}{2}\big(\frac{\partial^2u}{\partial x^2}\big)^{n-1}_i
+\Gamma(k)^2f(x_i,t_n),
\end{equation}
\begin{equation}\label{19}
v:=1+\alpha \Gamma(k),
\end{equation}
\begin{equation}\label{20}
w:=-\frac{\Gamma(k)^2}{2}.
\end{equation}
Substituting the approximate solution $U$ for $u$ and
putting the values (\ref{8}) and (\ref{10}) in (\ref{17})
yields the following difference equation with the variables $c_i,
i=-2,\ldots,N+2$.
\begin{align}\label{21}
&vU_{i}^{n+1}+w\big(U^{\prime\prime}\big)_{i}^{n+1}=\nonumber\\
&(v+20\frac{w}{h^2})c^{n+1}_{i+2}+(26v+40\frac{w}{h^2})c^{n+1}_{i+1}
+(66v-120\frac{w}{h^2})c^{n+1}_{i}\nonumber\\
&+(26v+40\frac{w}{h^2})c^{n+1}_{i-1}+(v+20\frac{w}{h^2})c^{n+1}_{i-2}=\Psi^n_{i},
\,\, i=0,\ldots,N,
\end{align}
where
\begin{eqnarray}\label{22}
\Psi^n_{i}:=\Big(2-\big(\beta \Gamma(k)\big)^2\Big)U^{n}_i
+\big(\alpha \Gamma(k)-1\big)U^{n-1}_i\nonumber\\
+\frac{\Gamma(k)^2}{2}\big(U^{\prime\prime}\big)^{n-1}_i
+\Gamma(k)^2f(x_i,t_n).
\end{eqnarray}
The system (\ref{21}) consists of $(N + 1)$ linear equations in $(N +
5)$ unknowns  $\widetilde{C}:=\{c_{-2},c_{-1},...,c_{N+1},c_{N+2}\}$. To obtain a
unique solution for $\widetilde{C}$ we
must use the boundary conditions. From the boundary conditions we can write
\begin{eqnarray}\label{23}
c^{n+1}_{-2}+26c^{n+1}_{-1}+66c^{n+1}_{0}+26c^{n+1}_{1}+c^{n+1}_{2}=g_0(t_{n+1}),\\\label{24}
c^{n+1}_{N-2}+26c^{n+1}_{N-1}+66c^{n+1}_{N}+26c^{n+1}_{N+1}+c^{n+1}_{N+2}=g_1(t_{n+1}),
\end{eqnarray}
and
\begin{eqnarray}\label{25}
\frac{1}{h}(-5c^{n+1}_{-2}-50c^{n+1}_{-1}+50c^{n+1}_{1}+5c^{n+1}_{2})=g_{2}(t_{n+1}),\\\label{26}
\frac{1}{h}(-5c^{n+1}_{N-2}-50c^{n+1}_{N-1}+50c^{n+1}_{N+1}+5c^{n+1}_{N+2})=g_{3}(t_{n+1}).
\end{eqnarray}

By using (\ref{23})-(\ref{26}), we obtain
\begin{align}\label{27}
&c^{n+1}_{-1}=-\frac{33}{8}c^{n+1}_0-\frac{9}{4}c^{n+1}_1
-\frac{1}{8}c^{n+1}_2+\frac{hg_2(t_{n+1})}{80}+\frac{g_0(t_{n+1})}{16},\\\label{28}
&c^{n+1}_{-2}=\frac{165}{4}c^{n+1}_0+\frac{65}{2}c^{n+1}_1
+\frac{9}{4}c^{n+1}_2-\frac{13hg_2(t_{n+1})}{40}
-\frac{5g_0(t_{n+1})}{8},\\\label{29}
&c^{n+1}_{N+1}=-\frac{33}{8}c^{n+1}_N-\frac{9}{4}c^{n+1}_{N-1}
-\frac{1}{8}c^{n+1}_{N-2}-\frac{hg_3(t_{n+1})}{80}+\frac{g_1(t_{n+1})}{16},\\\label{30}
&c^{n+1}_{N+2}=\frac{165}{4}c^{n+1}_N+\frac{65}{2}c^{n+1}_{N-1}
+\frac{9}{4}c^{n+1}_{N-2}+\frac{13hg_3(t_{n+1})}{40}-\frac{5g_1(t_{n+1})}{8}.
\end{align}

Hence we have the following system consists of $(N + 1)$ linear equations in $(N +
1)$ unknowns  $\{c_{0},c_{1},...,c_{N-1},c_{N}\}$.
The B-spline method in matrix form can be written as follows :
\begin{equation}\label{28}
AC=Q,
\end{equation}
where
\begin{align}\label{29}
&\mathbf{A}:=\nonumber\\
&\left(
\begin{array}{cccccccccccc}
\acute{c}-\frac{33\acute{b}}{8}+\frac{165\acute{a}}{4} & \frac{65\acute{a}}{2}-\frac{5\acute{b}}{4}&\frac{13\acute{a}}{4}-\frac{\acute{b}}{8}&0&\ldots&0 \\
\acute{b}-\frac{33\acute{a}}{8}& \acute{c}-\frac{9\acute{a}}{4} & \acute{b}-\frac{\acute{a}}{8} & a  &0&\ldots&0 \\
\acute{a}& \acute{b} & \acute{c}&\acute{b}&\acute{a}&\ldots&0 \\
\vdots&\ddots &\ddots&\ddots &\ddots&\ddots&\vdots&\\
0 & \ldots&\acute{a}& \acute{b}& \acute{c}&\acute{b}&\acute{a}\\
0&\ldots &0&\acute{a}& \acute{b}-\frac{33\acute{a}}{8}& \acute{c}-\frac{9\acute{a}}{4} & \acute{b}-\frac{\acute{a}}{8}\\
0&\ldots &0& 0&\acute{c}-\frac{33\acute{b}}{8}+\frac{165\acute{a}}{4} & \frac{65\acute{a}}{2}-\frac{5\acute{b}}{4}&\frac{13\acute{a}}{4}-\frac{\acute{b}}{8}\\
\end{array} \right),
\end{align}

\begin{equation}\label{30}
C:=\Big(c^{n+1}_{0},c^{n+1}_{1},\ldots,c^{n+1}_{N-1},c^{n+1}_{N}\Big)^T,
\end{equation}
and
\begin{equation}\label{31}
\mathbf{Q}:= \left(
\begin{array}{cccccccccccc}
\Psi^n(x_0)+(-\frac{bh}{80}+\frac{13ah}{40})g_2(t_{n+1})+(-\frac{b}{16}+\frac{5a}{8})g_0(t_{n+1})\\
\Psi^n(x_1)-\frac{ha}{80}g_2(t_{n+1})-\frac{a}{16}g_0(t_{n+1})\\
\Psi^n(x_2)\\
\vdots\\
\Psi^n(x_{N-2})\\
\Psi^n(x_{N-1})+\frac{ha}{80}g_3(t_{n+1})-\frac{a}{16}g_1(t_{n+1})\\
\Psi^n(x_N)+(\frac{bh}{80}-\frac{13ah}{40})g_3(t_{n+1})+(-\frac{b}{16}+\frac{5a}{8})g_1(t_{n+1}))\\
\end{array} \right),
\end{equation}
with
\begin{align}\label{32}
&\acute{a}:=v+20\frac{w}{h^2},             \\\label{33}
&\acute{b}:=26v+40\frac{w}{h^2},           \\\label{34}
&\acute{c}:=66v-120\frac{w}{h^2}.
\end{align}
The computer algebra system $Mathematica$-$9$ is used for solving the
system (\ref{28}). To start any computation, it is necessary to know the value of $u$ at the nodal points of first
time level, that is, at $t =k$. A Taylor series expansion
at $t =k$ may be written as
\begin{equation}\label{35}
u(x,k)=f_{0}(x)+kf_{1}(x)+\frac{k^2}{2}\big(f(x,0)-\beta^2f_{0}(x)-2\alpha f_{1}(x)+\frac{\partial
f_{0}(x)}{\partial x}\big)+\mathcal{O}(k^3).
\end{equation}
%-------------------------------Convergence analysis-------------------------------------
\vskip .2cm
\section{Convergence analysis}
%-------------------------------------------Theorem1 ---------------------------------------
\begin{thm}\label{th0}
Suppose that  $u(x,t)$ be the exact
solution of (\ref{1}) and $u(x,t)\in \mathcal{C}^5[a,b]$ also $|\frac{\partial^5u(x,t)}{\partial x^5}|\leq L$ and $U(x,t)$ be the numerical approximation  by our methods, then we can write
\begin{equation}
\parallel u(x,t)-U(x,t) \parallel_{\infty}=\mathcal{O}(h^2+k).
\end{equation}
\end{thm}
Before we prove, we recall following theorem and lemma.
%-------------------------------------------Theorem1 ---------------------------------------
\begin{thm}\label{th1}
Suppose that $f(x)\in \mathcal{C}^5[a,b]$. Then for the unique quintic spline $S(x)$
associated with $f$, we have
\begin{equation}\label{32}
\parallel f^{(j)}-S^{(j)} \parallel_{\infty}\leq K_j\omega_{5}(h) h^{4-j}\ ,\
j=0,1,2,3.
\end{equation}
where $\omega_{5}(h)$ denotes the modulus of continuity of $f^{(5)}$ and the coefficients $\lambda_j$
are independent of f and h.
\end{thm}
\begin{proof}
For the proof see \cite{11}.
\end{proof}
%________________________________________________________
\begin{rem}
By using Theorem \ref{th1} and definition of the modulus of continuity, we can say that if $|f^{(5)}(x)|\leq L$,
we can write (\ref{32}) as
\begin{equation}\label{33}
\parallel f^{(j)}-S^{(j)} \parallel_{\infty}\leq\lambda_jLh^{4-j}\ ,\
j=0,1,2,3.
\end{equation}
\end{rem}
%-----------------------------------------------------------Lemma----------------------------------
\begin{lem}\label{lm1}
For the B-splines $\{B_{-2},\cdots,B_{N+2}\}$ we have
the following inequality:
\begin{equation}
|\sum_{i=-2}^{N+2}B_i(x)| \leq 186 , \ \,\ (a \leq  x \leq b).
\end{equation}
\end{lem}
\begin{proof}
From the real analysis we have
$|\sum_{i=-2}^{N+2}B_i(x)| \leq \sum_{i=-2}^{N+2}|B_i(x)|$.
If $x=x_i,$ $i=1,\ldots,N,$ then, we have
\begin{align}
|\sum_{i=-2}^{N+2}B_i(x)| = 120\leq 186,
\end{align}
and if $x_{i-1}\leq x \leq x_{i}$, then, we can write
\begin{align}\nonumber
|\sum_{i=-2}^{N+2}B_i(x)| &\leq\mid B_{i-3}(x) \mid+ \mid B_{i-2}(x)
+\mid B_{i-1}(x) \mid+\mid B_{i}(x) \mid\mid\nonumber\\
&+\mid B_{i+1}(x)\mid+\mid B_{i+2}(x)\mid\nonumber
\leq
1+26+66+66+26+1\leq
186
\end{align}
\end{proof}
%---------------------------------------------------------Proof----------------------------------
Now we prove theorem \ref{th0}.
\begin{proof}
Suppose that $\varepsilon _{i}=u(t_i)-U^i$ be
the local truncation error for (\ref{16}) at the $i$th.
By using the truncation error, we can write
\begin{equation}\label{c1}
\mid\varepsilon_{i}\mid \leq \varrho_{i}k^2 \ ,\ i \geq 2.
\end{equation}
In addition we have $\mid\varepsilon_{1}\mid \leq \varrho_{1}k^3$.
To continue we assume that  $e_{n+1}$ be the global error in time discretizing process and
$\varrho=max\{{\varrho_1,...,\varrho_n}\}$. We can write the following global error estimate at $n+1$ level
\begin{equation}\label{c2}
e_{n+1}=\sum_{i=1}^{n}
{\varepsilon_i},~~~~(k\leq T/n),
\end{equation}
with the help of (\ref{c1})-(\ref{c2}), we can write
\begin{equation}\label{c3}
\mid e_{n+1}\mid=\mid \sum^{n}_{i=1}\varepsilon_i  \mid \leq
n\varrho k^2\leq n\varrho \frac{T}{n}k=\rho k,
\end{equation}
where  $ \rho= \varrho T.$\\
Now at the $(n + 1)$th time step we assume that $u(x)$ be the exact solution of (\ref{17})
and $U(x)=\sum_{i=-2}^{N+2}
c_iB_i(x)$ be the B-spline approximation to
$u(x)$. Also we assume
that $S^{*}(x)=\sum_{i=-2}^{N+2}c^*_iB_i(x)$ be
the unique spline interpolate to the exact solution. In order to derive a bound for $\parallel
u(x)-U(x)\parallel_\infty$, we need to estimate the
$\parallel u(x)-S^*(x)\parallel_\infty$ and $\parallel
S^*(x)-U(x)\parallel_\infty$. Now we substituting $S^*(x)$ in (\ref{17})
the we get the following result
\begin{equation}\label{c4}
AC^*=Q^*,
\end{equation}
With considering (\ref{28}) and (\ref{c4}), we get
\begin{equation}\label{c5}
A(C^*-C)=(Q^*-Q).
\end{equation}
From (\ref{21}), we can writte
\begin{align}\label{c6}
|\Psi^*_{i}-\Psi_{i}|\leq v\mid S^*(x_i)-U(x_i)\mid+w\mid S^{*''}(x_i)-U^{''}(x_i) \mid.
\end{align}
By using (\ref{c6}) and Theorem \ref{th1}, we can write
\begin{equation}\label{c7}
\parallel Q^*-Q\parallel_{\infty}\leq M_1h^2,
\end{equation}
where $M_1=v\lambda_0Lh^2+w\lambda_2L$.
In this step from (\ref{c5}), we can write
\begin{equation}\label{c8}
(C^*-C)=A^{-1}(Q^*-Q).
\end{equation}
By taking the infinity norm from (\ref{c8}) and applying (\ref{c7}), we get
\begin{equation}\label{c9}
\parallel C^*-C\parallel_{\infty}\leq \parallel A^{-1}\parallel_{\infty}\parallel Q^*-Q \parallel_{\infty} \leq M_1h^2\parallel
A^{-1}\parallel_{\infty}.
\end{equation}
By using the theory of matrices, we can write
\begin{equation}\label{c10}
\sum_{i=1}^{N+1}a^{-1}_{ki}\eta_i =1,
\end{equation}
where $a^{-1}_{ki}$ are the elements of $A^{-1}$  and  $\eta_i$ $(1\leq i\leq N+1)$ is the summation of the
$i$th row of the matrix A. As a result  we
can write
\begin{equation}\label{c11}
\|A^{-1}\|_{\infty}=\sum_{i=1}^{N+1}|a^{-1}_{ki}|\leq\frac{1}{\min_{1 \leq
i\leq N}\eta_i}\leq\frac{1}{\Lambda},
\end{equation}
where $\Lambda$ is is constant. Following result is obtained by
substituting (\ref{c11}) into (\ref{c9}), we get
\begin{equation}\label{c12}
\parallel C^*-C \parallel_{\infty}\leq \frac{M_1h^2}{\Lambda} \leq
M_2h^2,
\end{equation}
where $M_2=\frac{M_1}{\Lambda}$ is constant.
Considering the B-spline collocation approximation
and the computed spline approximation, we can write:
\begin{equation}\label{c14}
S^*(x)-U(x)=\sum_{i=-2}^{N+2}(c^*_i-c_i)B_i(x),
\end{equation}
taking norm from (\ref{c14}) and by using (\ref{c12}) and lemma \ref{lm1}, we obtain
\begin{equation}\label{c15}
\parallel S^*(x)-U(x) \parallel_{\infty}= \parallel \sum_{i=-2}^{N+2}(c^*_i-c_i)B_i(x)\parallel_{\infty} \leq \big| \sum_{i=-2}^{N+2}B_i(x)\big| \parallel C^*-C \parallel_{\infty} \leq
186M_2h^2.
\end{equation}
Also from Theorem \ref{th1} we can write
\begin{equation}\label{c16}
\parallel u-S^*(x) \parallel_{\infty} \leq \lambda_0Lh^4,
\end{equation}
and therefore with helping (\ref{c15}) and (\ref{c16}) we get
\begin{equation}\label{c17}
\parallel  u-U(x) \parallel_{\infty}  \leq\varpi h^2,
\end{equation}
where $\varpi=\lambda_0Lh^2+186M_2.$
\end{proof}

%**************************************Numerical examples***********************************************
\section{Numerical examples}
\indent \hskip .65cm In order to illustrate the performance of the
quintic B-spline collocation method in solving the One-dimensional
hyperbolic telegraph  equation and justify the accuracy and
efficiency of the present method, we consider the following
examples. To show the efficiency of the present method for our
problem in comparison with the exact solution, we report the RMS
error, $L_\infty$ and $L_2$ using formulae
\begin{align}\nonumber
&RMS= \frac{(\sum^{N}_{i=1}|u(x_i,t)-U_n(x_i)|^2)^{\frac{1}{2}}}{
N^{\frac{1}{2}}},\\
&L_\infty=\max_i|U_n(x_i)-u(x_i,t)|,\,
L_2=h|\sum^{N}_{i=1}(u(x_i,t)-U_n(x_i)|^2,\nonumber
\end{align}
where $U(x,t)$ denotes numerical solution and $u(x,t)$ denotes
analytical solution.
%------------------------- Example 1------------------------------------------------------------------
\vskip .3cm \noindent \textbf{Example 1.} Consider the hyperbolic
telegraph equation (\ref{1}) with $\alpha=\pi,\beta=\pi$, in the interval
$[0,1]$. In this case we have
$f(x,t)=\pi ^2 \sin (\pi  x) (\sin (\pi  t)+2 \cos (\pi  t))$. The analytical solution
given  by $u(x,t)=\sin (\pi  t) \sin (\pi  x)$. The boundary conditions and
the initial conditions are taken from the exact solution. Table \ref{table:t11} shows
the absolute error between the analytical solution and the numerical
solution at different points for $t=0.5$. Table \ref{table:t12} shows
the $L_2$ errors at different partitions.
The graph of the solution is given in Figure {\color{blue}1}
. Also, Figure {\color{blue}2}
shows that the solution obtained by our method is close
to the exact solution
\begin{center}
\begin{table}[ht!]
\caption{A comparison of absolute errors
of Example 1 at different points with $h=1/100,k=1/200$.} % title of Table
\label{table:t11} % is used to refer this table in the text
\centering % used for centering table
\begin{tabular}{l l l l l l l l l l l l} % centered columns (4 columns)
\hline
Method &\multicolumn{2}{|c|}{present method}&\multicolumn{3}{|c}{method in \cite{12}}\\
\hline
$x$&$\Gamma:2sin(\frac{k}{2})$ &$\Gamma:k$ & $\eta=\frac{1}{60},\gamma=\frac{1}{2}$
grid&&\\
\hline
\hline
0.2     &3.23677$\times10^{-5}$   &3.26277$\times10^{-5}$     &5.858898718658607$\times10^{-1}$&  \\
0.4     &5.32737$\times10^{-5}$    &5.32737$\times10^{-5}$    &9.479897263432836$\times10^{-1}$& \\
0.6     &5.32737$\times10^{-5}$   &5.32737$\times10^{-5}$     &9.479897263432840$\times10^{-1}$&  \\
0.8     &5.32737$\times10^{-5}$   &3.26277$\times10^{-5}$     &5.858898718658610$\times10^{-1}$&  \\
\hline
\hline
\end{tabular}
\end{table}
\end{center}
\begin{center}
\begin{table}[ht!]
\caption{A comparison of $L_{2}$ errors
of Example 1  at different partitions.} % title of Table
\label{table:t12} % is used to refer this table in the text
\centering % used for centering table
\begin{tabular}{l l l l l l l l l l l l} % centered columns (4 columns)
\hline
Partitions&\multicolumn{2}{|c|}{N=100,k=0.01}&\multicolumn{3}{|c}{N=400,k=0.001}\\
\hline
$Time$&$\Gamma:2sin(\frac{k}{2})$ &$\Gamma:k$ &$\Gamma:2sin(\frac{k}{2})$ & $\Gamma:k$ \\
\hline
\hline
0.5   &1.5799$\times10^{-5}$    &1.58168$\times10^{-5}$   &1.57942$\times10^{-7}$   &7.91938$\times10^{-8}$ &  \\
1     &6.94858$\times10^{-6}$   &6.61076$\times10^{-6}$   &6.94312$\times10^{-8}$   &3.33013$\times10^{-8}$ &  \\
1.5   &1.50368$\times10^{-5}$   &1.51490$\times10^{-5}$   &1.50334$\times10^{-7}$   &7.57557$\times10^{-8}$ &  \\
2     &7.25141$\times10^{-6}$   &6.89868$\times10^{-6}$   &7.24547$\times10^{-8}$       &3.4743$\times10^{-8}$ &  \\
\hline
\hline
\end{tabular}
\end{table}
\end{center}
\begin{center}
\begin{figure}\label{fig:1}
\centering
\includegraphics[]{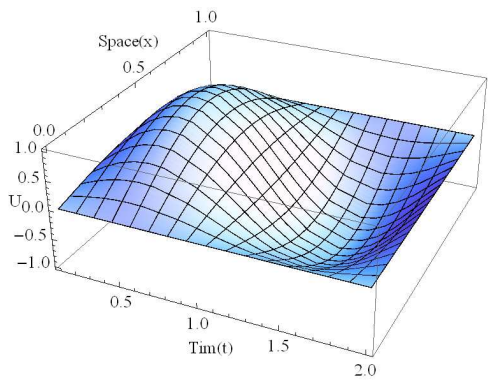}
\emph{\caption{The graph of the solution for Example 1 with $\Gamma(k)=k, N=200$ and k=0.01.}}
\end{figure}
\end{center}
\begin{center}
\begin{figure}
\centering
\includegraphics[]{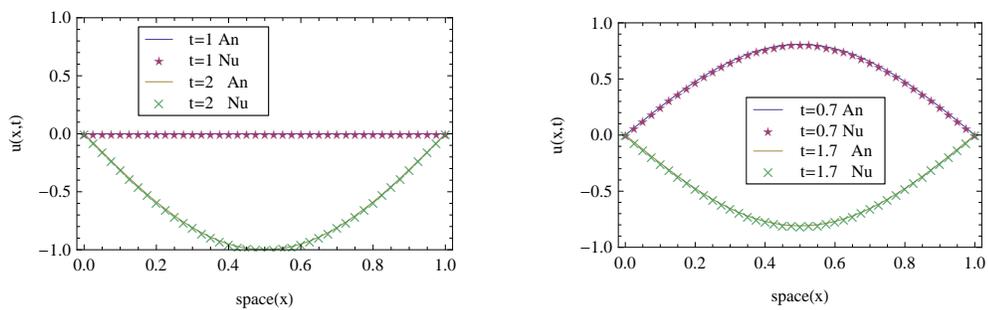}
\emph{\caption{Comparisons between numerical and analytical
solutions with $\Gamma(k)=k$ (left) and $\Gamma(k)=2sin(k/2)$ (right) for Example 1 at different times with $N=200,k=0.01$.}}
\end{figure}
\end{center}
%================== Example 2
\vskip .3cm \noindent \textbf{Example 2.}
We consider the hyperbolic
telegraph equation (\ref{1}) with $f(x,t)=(3-4\alpha+\beta^2)\exp(-2t)\sinh(x)$
and the analytical solution
$u(x,t)=\exp(-2t)\sinh(x)$ , in the interval
$[0,1]$. The boundary conditions and
the initial conditions are taken from the exact solution.  Tables \ref{table:t21} and \ref{table:t22} give a
comparison between the $L_{\infty}$ errors found by our method and the
method in \cite{7}. Also Table \ref{table:t22} shows $RMS$ and $L_2$ errors. Figure {\color{blue}3}, shows absolute error for
different values of time with $N=200, k=0.001$.
\begin{center}
\begin{table}[ht!]
\caption{A comparison of $L_{\infty}$ errors
of Example 2 for $\alpha=20$, $\beta=10$ at different time and $N=21,k=0.01$.} % title of Table
\label{table:t21} % is used to refer this table in the text
\centering % used for centering table
\begin{tabular}{l l l l l l l l l l l l} % centered columns (4 columns)
\hline
Method &\multicolumn{2}{|c|}{present method}&\multicolumn{3}{|c}{method in \cite{7}}\\
\hline
$Time$&$\Gamma:2sin(\frac{k}{2})$ &$\Gamma:k$ & Uniform Grid & Nonuniform
grid\\
\hline
\hline
0.5   &2.73131$\times10^{-6}$  &2.59359$\times10^{-6}$  &2.25640$\times10^{-4}$   &2.22066$\times10^{-4}$ &  \\
1     &1.5714$\times10^{-6}$  &1.48998$\times10^{-6}$   &5.75205$\times10^{-3}$   &1.41294$\times10^{-4}$ &  \\
1.5   &7.01914$\times10^{-7}$  &6.65254$\times10^{-7}$  &2.43143$\times10^{-2}$   &6.93111$\times10^{-5   }$ &  \\
2     &2.8573$\times10^{-7}$  &2.70754$\times10^{-7}$    &-----------------     &2.27829$\times10^{-5}$ &  \\
\hline
\hline
\end{tabular}
\end{table}
\end{center}
\begin{center}
\begin{table}[ht!]
\caption{A comparison of $L_{\infty}$ errors
of Example 2 for $\alpha=20$, $\beta=10$ at different time and $N=21,k=0.0001$.} % title of Table
\label{table:t22} % is used to refer this table in the text
\centering % used for centering table
\begin{tabular}{l l l l l l l l l l l l} % centered columns (4 columns)
\hline
Method &\multicolumn{2}{|c|}{present method}&\multicolumn{3}{|c}{method in \cite{7}}\\
\hline
$Time$&$\Gamma:2sin(\frac{k}{2})$ &$\Gamma:k$ & Uniform Grid & Nonuniform
grid\\
\hline
\hline
0.5   &1.9517$\times10^{-9}$  &1.92002$\times10^{-9}$  &2.23874$\times10^{-6}$   &2.21693$\times10^{-6}$ &  \\
1     &1.08081$\times10^{-9}$   &1.06278$\times10^{-9}$   &1.72404$\times10^{-5}$   &1.22688$\times10^{-6}$ &  \\
1.5   &4.75261$\times10^{-10}$  &4.67199$\times10^{-10}$  &4.14630$\times10^{-3}$   &6.73102$\times10^{-7}$ &  \\
2     &1.91059$\times10^{-10}$   &1.87808$\times10^{-10}$    &-----------------       &2.66660$\times10^{-7}$ &  \\
\hline
\hline
\end{tabular}
\end{table}
\end{center}
\begin{center}
\begin{table}[ht!]
\caption{ $L_{2}$ and $RMS$ errors
of Example 2 for $\alpha=20$, $\beta=10$ at different time and $N=21,k=0.01$.} % title of Table
\label{table:t23} % is used to refer this table in the text
\centering % used for centering table
\begin{tabular}{l l l l l l l l l l l l} % centered columns (4 columns)
\hline
$\Gamma$ &\multicolumn{2}{|c|}{$2sin(\frac{k}{2})$}&\multicolumn{3}{|c}{$k$}\\
\hline
$Time$&$L_2$ &$RMS$ & $L_2$ &$RMS$
\\
\hline
\hline
0.2&2.36511$\times10^{-6}$   &1.08383$\times10^{-5}$   &2.25651$\times10^{-6}$  &1.03406$\times10^{-5}$&     \\
0.4 &2.82003$\times10^{-6}$   &1.2923$\times10^{-5}$   &2.67987$\times10^{-6}$  &1.22807$\times10^{-5}$&    \\
0.8 &2.0589$\times10^{-6}$    &9.43506$\times10^{-6}$  &1.95292$\times10^{-6}$  &8.94939$\times10^{-6}$ &    \\
1 &1.5714$\times10^{-6}$    &7.20105$\times10^{-6}$   &1.48999$\times10^{-6}$  &6.82797$\times10^{-6}$ &   \\
\hline
\hline
\end{tabular}
\end{table}
\end{center}
\begin{center}
\begin{figure}
\centering
\includegraphics[]{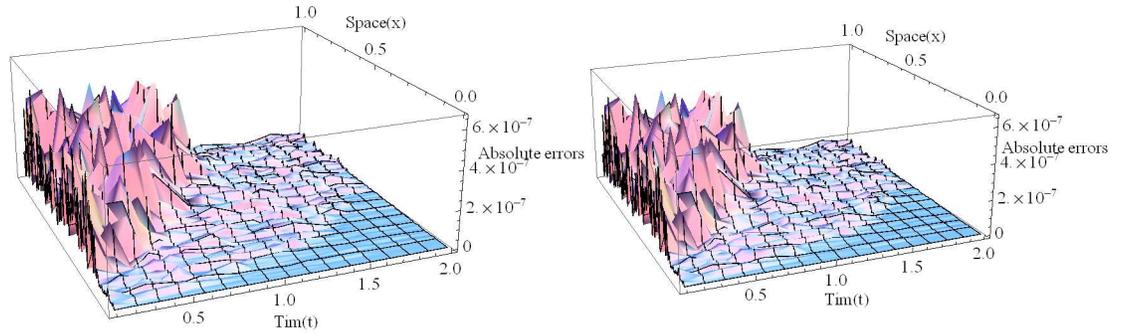}
\emph{\caption{Absolute errors with $\Gamma(k)=k$ (left) and $\Gamma(k)=2sin(k/2)$ (right) for Example 2 with $k=0.001$, $N=200$ and  $\alpha=20, \beta=10$.}}
\end{figure}
\end{center}
%======================================================
%------------------------- Example 3------------------------------------------------------------------
\vskip .3cm \noindent \textbf{Example 3.} In this example we Consider the hyperbolic
telegraph equation (\ref{1}) with $\alpha=10,\beta=5$ in $x\in[0,1]$ and $f(x,t)=-2\alpha \sin(t)\sin(x)+\beta^2\cos(t)\cos(x)$. The
exact solution for this case is $u(x,t)=\cos(t)\sin(x)$. The
boundary conditions and the initial conditions are taken from the
exact solution. In order to compare the solutions with \cite{7}, we
have taken $k=0.001$ and $N=21$. Table \ref{table:t31} gives a
comparison between the  $L_{\infty}$ error found by our method and by
method in \cite{7}.
Table \ref{table:t32} shows  $L_2$ in different partitions.
Table \ref{table:t33} shows $RMS$ and $L_2$ errors. Figure {\color{blue}4}, shows absolute error for
different values of time with $N=50, k=0.001$.
\begin{center}
\begin{table}[ht!]
\caption{A comparison of $L_{\infty}$ errors
for Example 3  at different time and $N=21,k=0.001$.} % title of Table
\label{table:t31} % is used to refer this table in the text
\centering % used for centering table
\begin{tabular}{l l l l l l l l l l l l} % centered columns (4 columns)
\hline
Method &\multicolumn{2}{|c|}{present method}&\multicolumn{3}{|c}{method in \cite{7}}\\
\hline
$Time$&$\Gamma:2sin(\frac{k}{2})$ &$\Gamma:k$ & Uniform Grid & Nonuniform
grid\\
\hline
\hline
0.5   &7.59175$\times10^{-9}$  &5.35357$\times10^{-9}$  &1.67216$\times10^{-5}$   &1.28551$\times10^{-5}$ &  \\
1     &3.5984$\times10^{-9}$   &2.39252$\times10^{-8}$   &4.71302$\times10^{-4}$   &3.20793$\times10^{-5}$ &  \\
1.5   &1.61922$\times10^{-8}$  &4.2016$\times10^{-8}$  &8.62796$\times10^{-4}$   &5.71454$\times10^{-5}$ &  \\
2     &2.49884$\times10^{-8}$   &5.20784$\times10^{-8}$   &-----------------        &6.82827$\times10^{-5}$ &  \\
\hline
\hline
\end{tabular}
\end{table}
\end{center}
\begin{center}
\begin{table}[ht!]
\caption{A comparison of $L_{2}$ errors
of Example 3  at different partitions.} % title of Table
\label{table:t32} % is used to refer this table in the text
\centering % used for centering table
\begin{tabular}{l l l l l l l l l l l l} % centered columns (4 columns)
\hline
Partitions&\multicolumn{2}{|c|}{N=200,k=0.01}&\multicolumn{3}{|c}{N=400,k=0.001}\\
\hline
$Time$&$\Gamma:2sin(\frac{k}{2})$ &$\Gamma:k$ &$\Gamma:2sin(\frac{k}{2})$ & $\Gamma:k$ \\
\hline
\hline
0.5   &3.69769$\times10^{-8}$    &2.46374$\times10^{-8}$   &2.65293$\times10^{-10}$   &1.73258$\times10^{-10}$ &  \\
1     &1.40829$\times10^{-8}$   &1.16929$\times10^{-7}$   &9.95706$\times10^{-11}$   &8.27160$\times10^{-10}$ &  \\
1.5   &8.1959$\times10^{-8}$   &2.09527$\times10^{-7}$   &5.80334$\times10^{-10}$   &1.48296$\times10^{-10}$ &  \\
2     &1.28423$\times10^{-7}$   &2.61864$\times10^{-7}$   &9.09458$\times10^{-10}$   &1.85373$\times10^{-9}$ &  \\
\hline
\hline
\end{tabular}
\end{table}
\end{center}
\begin{center}
\begin{center}
\begin{table}[ht!]
\caption{ $L_{2}$ and $RMS$ errors
of Example 3 at different time and $N=80,k=0.05$.} % title of Table
\label{table:t33} % is used to refer this table in the text
\centering % used for centering table
\begin{tabular}{l l l l l l l l l l l l} % centered columns (4 columns)
\hline
$\Gamma$ &\multicolumn{2}{|c|}{$2sin(\frac{k}{2})$}&\multicolumn{3}{|c}{$k$}\\
\hline
$Time$&$L_2$ &$RMS$ & $L_2$ &$RMS$
\\
\hline
\hline
0.5&1.37362$\times10^{-6}$   &1.2286$\times10^{-5}$   &9.94623$\times10^{-7}$  &8.89618$\times10^{-6}$&     \\
1 &5.5496$\times10^{-7}$   &4.96372$\times10^{-6}$     &4.61873$\times10^{-6}$  &4.13112$\times10^{-5}$&    \\
1.5 &3.22627$\times10^{-6}$    &2.88566$\times10^{-5}$  &8.26127$\times10^{-6}$  &7.38911$\times10^{-5}$ &    \\
2 &5.05609$\times10^{-6}$    &4.52231$\times10^{-5}$    &1.03189$\times10^{-5}$  &9.22947$\times10^{-5}$ &   \\
\hline
\hline
\end{tabular}
\end{table}
\end{center}

\begin{figure}
\centering
\includegraphics[]{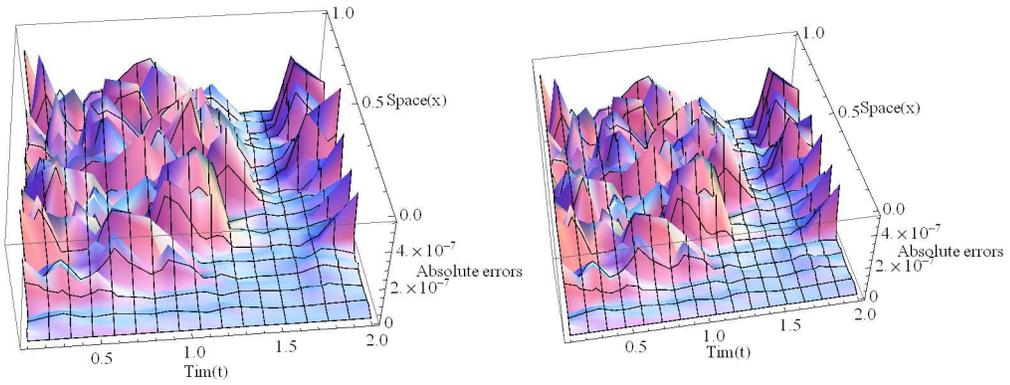}
\emph{\caption{Absolute errors with $\Gamma(k)=k$ (left) and $\Gamma(k)=2sin(k/2)$ (right) for Example 3 with $k=0.001$, $N=50$.}}
\end{figure}
\end{center}
\section{Conclusion}
The quintic B-spline collocation method is used to solve the
one-dimensional hyperbolic telegraph equation with initial and
boundary conditions. The convergence analysis of the method is shown. The numerical solutions are compared with the
exact solution by finding the RMS ,$L_2$ and $L_\infty$ errors.The
numerical results given in the previous section demonstrate the
good accuracy of the scheme proposed in this research.
%***************************************knowledgement**************************************************************
%****************************************References**************************************************************

\end{document}